\documentclass{amsart}
\usepackage{amssymb}
\usepackage{amsmath}
\usepackage{amsfonts, latexsym}
\usepackage{color,graphics}
\usepackage{graphicx}

\newtheorem{theorem}{Theorem}[section]
\newtheorem{proposition}[theorem]{Proposition}
\newtheorem{lemma}[theorem]{Lemma}

\newtheorem{definition}[theorem]{Definition}
\newtheorem{example}[theorem]{Example}
\newtheorem{remark}[theorem]{Remark}
\newtheorem{assumption}[theorem]{Assumption}

\begin{document}
\author{Hee Sun Jung$^1$ and Ryozi Sakai$^2$}
\address{$^{1}$Department of Mathematics Education, Sungkyunkwan University,
Seoul 110-745, Republic of Korea.}
\email{hsun90@skku.edu}
\address{$^{2}$Department of Mathematics, Meijo University, Nagoya 468-8502, Japan.}
\email{ryozi@crest.ocn.ne.jp}

\title[]{$L_p$-Convergence of higher order Hermite or Hermite-Fej\'er interpolation polynomials with
exponential-type weights}
\date{\today}
\maketitle

\begin{abstract}
Let $\mathbb{R}=(-\infty,\infty)$,
and let $Q\in C^1(\mathbb{R}): \mathbb{R}\rightarrow \mathbb{R^+}=[0,\infty)$ be an even function,
which is an exponent. We consider the weight $w_\rho(x)=|x|^{\rho} e^{-Q(x)}$, $\rho\geqslant 0$, $x\in \mathbb{R}$,
and then we can construct the orthonormal polynomials $p_{n}(w_\rho ^2;x)$ of degree n for $w_\rho ^2(x)$.
In this paper we obtain $L_p$-convergence theorems of even order Hermite-Fej\'er interpolation polynomials
at the zeros $\left\{x_{k,n,\rho}\right\}_{k=1}^n$ of $p_{n}(w_\rho ^2;x)$.
\end{abstract}

MSC; 41A05, 41A10\\
Keywords; higher order Hermite-Fej\'er interpolation polynomials

\setcounter{equation}{0}
\section{Introduction}

Let $\mathbb{R}=(-\infty,\infty)$,
and let $Q\in C^1(\mathbb{R}): \mathbb{R}\rightarrow \mathbb{R^+}=[0,\infty)$ be an even function.
Consider the weight $w(x)=\exp(-Q(x))$ and define for $\rho>-\frac{1}{2}$,
\begin{equation}\label{eq1.1}
   w_\rho(x):=|x|^{\rho} w(x), \quad  x\in \mathbb{R}.
\end{equation}
Suppose that $\int_0^\infty x^nw_\rho^2(x)dx<\infty$ for all $n=0,1,2,\ldots$.
Then, we can construct the orthonormal polynomials $p_{n,\rho}(x)=p_{n}(w_\rho ^2;x)$ of degree $n$ for $w_\rho ^2(x)$, that is,
\begin{equation*}
   \int_{-\infty}^\infty p_{n,\rho}(x)p_{m,\rho}(x)w_\rho ^2(x)dx=\delta_{mn}  (\textrm{Kronecker's delta}).
\end{equation*}
Denote
\begin{equation*}
   p_{n,\rho}(x)=\gamma_{n}x^n+\ldots,  \gamma_{n}=\gamma_{n,\rho}>0,
\end{equation*}
and the zeros of $p_{n,\rho}(x)$ by
\begin{equation*}
   -\infty<x_{n,n,\rho}<x_{n-1,n,\rho}<\ldots<x_{2,n,\rho}<x_{1,n,\rho}<\infty.
\end{equation*}
Let $\mathcal{P}_n$ denote the class of polynomials with degree at most $n$.
For $f\in C(\mathbb{R})$ we define the higher order Hermite-Fej\'er interpolation polynomial
$L_n(\nu,f;x)$ based at the zeros $\left\{x_{k,n,\rho}\right\}_{k=1}^n$ as follows:
\begin{equation*}
L_n^{(i)}(\nu,f;x_{k,n,\rho})=\delta_{0,i}f(x_{k,n,\rho})
\quad \textrm{for} \quad    k=1,2,\ldots,n    \quad \textrm{and} \quad    i=0,1,\ldots,\nu-1.
\end{equation*}
$L_n(1,f;x)$ is the Lagrange interpolation polynomial, $L_n(2,f;x)$ is the ordinary Hermite-Fej\'er interpolation polynomial,
and $L_n(4,f;x)$ is the Krilov-Stayermann polynomial.
The fundamental polynomials $h_{k,n,\rho}(\nu;x)\in \mathcal{P}_{\nu n-1}$
for the higher order Hermite-Fej\'er interpolation polynomial $L_n(\nu,f;x)$ are defined as follows:
\begin{equation}\label{eq1.2}
\begin{array}{l}
h_{k,n,\rho}(\nu;x)=l_{k,n,\rho}^{\nu}(x)   \sum_{i=0}^{\nu -1}e_i(\nu,k,n)(x-x_{k,n,\rho})^i, \\
l_{k,n,\rho}(x)=\frac{p_n(w_{\rho}^2;x)}{(x-x_{k,n,\rho})p'_n(w_{\rho}^2;x_{k,n,\rho})},\\
h_{k,n,\rho}(\nu;x_{p,n,\rho})=\delta_{k,p},  \quad  h_{k,n,\rho}^{(i)}(\nu;x_{p,n,\rho})=0, \\
\quad         k,p=1,2,\ldots,n,   i=1,2,\ldots,\nu -1.
\end{array}
\end{equation}
Using them, we can write $L_n(\nu,f;x)$ as follows:
\begin{equation*}
  L_n(\nu,f;x)= \sum_{k=1}^nf(x_{k,n,\rho})h_{k,n,\rho}(\nu;x).
\end{equation*}
Furthermore, we extend the operator $L_n(\nu,f;x)$. Let $l$ be a non-negative integer,
and let $l\le \nu -1$. For $f \in C^{l}(\mathbb{R})$
we define the $(l,\nu)$-order Hermite-Fej\'er interpolation polynomials $L_n(l,\nu,f;x)\in \mathcal{P}_{\nu n-1}$
as follows: For each $k=1,2,\ldots,n,$
\begin{equation*}
\begin{array}{c}
L_n(l,\nu,f;x_{k,n,\rho})=f(x_{k,n,\rho}),\\
L_n^{(j)}(l,\nu,f;x_{k,n,\rho})=f^{(j)}(x_{k,n,\rho}),  \quad j=1,2,\ldots,l,\\
L_n^{(j)}(l,\nu,f;x_{k,n,\rho})=0,   \quad j=l+1,l+2,\ldots,\nu -1.
\end{array}
\end{equation*}
Especially, $L_n(0,\nu,f;x)$ is equal to $L_n(\nu,f;x)$, and for each $P\in \mathcal{P}_{\nu n-1}$
we see $L_n(\nu-1,\nu,P;x)=P(x)$.
The fundamental polynomials $h_{s,k,n,\rho}(l,\nu;x)\in \mathcal{P}_{\nu n-1},   k=1,2,\ldots,n$, of $L_n(l,\nu,f;x)$ are defined by
\begin{equation}\label{eq1.3}
\begin{array}{c}
h_{s,k,n,\rho}(l,\nu;x)=l_{k,n,\rho}^{\nu}(x)   \sum_{i=s}^{\nu -1}e_{s,i}(\nu,k,n)(x-x_{k,n,\rho})^i,\\
h_{s,k,n,\rho}^{(j)}(l,\nu;x_{p,n,\rho})=\delta_{s,j}\delta_{k,p},   j,s=0,1,\ldots,\nu -1,  p=1,2,\ldots,n.
\end{array}
\end{equation}
Then we have
\begin{equation*}
L_n(l,\nu,f;x)=   \sum_{k=1}^n    \sum_{s=0}^lf^{(s)}(x_{k,n,\rho})h_{s,k,n,\rho}(l,\nu;x).
\end{equation*}
For the ordinary Hermite and Hermite-Fej\'er interpolation polynomial $L_n(1,2,f;x)$, $L_n(2,f;x)$ and the related approximation process,
 Lubinsky \cite{[15]} gave some interesting convergence theorems.

In \cite{[7]} we obtained uniform convergence theorems
with respect to the interpolation polynomials $L_n(\nu,f;x)$ and $L_n(l,\nu,f;x)$ with even integer $\nu$.
In this paper we will give the $L_p$-convergence theorems of $L_n(\nu,f;x)$ and $L_n(l,\nu,f;x)$ with even order $\nu$.
If we consider the higher order Hermite-Fej\'er interpolation polynomial $L_n(\nu,f;x)$ on a finite interval,
then we can see a remarkable difference between the cases of an odd number $\nu$ and of an even number $\nu$,
for example, as between the Lagrange interpolation polynomial $L_n(1,f;x)$
and the Hermite-Fej\'er interpolation polynomial $L_n(2,f;x)$ (\cite{[17]}-\cite{[22]}). We can also see a similar phenomenon
in the cases of the infinite intervals (\cite{[7]}-\cite{[14]}). In this paper, we consider the even case in $L_p$-norm.
We will discuss the odd case in $L_p$-norm elsewhere.

Here, we give the class of weights which is treated in this paper.
We say that $f: \mathbb{R}\rightarrow \mathbb{R^+}$ is quasi-increasing if there exists $C>0$
such that $f(x)\le Cf(y),  0<x<y$.
In the following, we introduce the class of weights defined in \cite{[16]}.
\begin{definition}[see \cite{[16]}]\label{Definition1.1}
Let $Q:  {\mathbb{R}}\rightarrow {\mathbb{R}}^+$ be a continuous even function satisfying the following properties:
\item[\,\,\,(a)] $Q'(x)$ is continuous in ${\mathbb{R}}$ and $Q(0)=0$.
\item[\,\,\,(b)] $Q''(x)$ exists and is positive in ${\mathbb{R}}\setminus\{0\}$.
\item[\,\,\,(c)] $  \lim_{x\rightarrow \infty}Q(x)=\infty.$
\item[\,\,\,(d)] The function
\begin{equation*}
T(x):=\frac{xQ'(x)}{Q(x)}, \quad x\neq 0
\end{equation*}
is quasi-increasing in $(0,\infty)$, with
\[
T(x)\geqslant \Lambda>1, \quad x\in {\mathbb{R}}^+\setminus\{0\}.
\]
\item[\,\,\,(e)] There exists $C_1>0$ such that
\begin{equation*}
\frac{Q''(x)}{|Q'(x)|}\le C_1\frac{|Q'(x)|}{Q(x)},  \quad a.e. \quad x\in {\mathbb{R}}\setminus\{0\}.
\end{equation*}
Then we say that $w=\exp(-Q)$ is in the class $\mathcal{F}(C^2)$.
Besides, if there exists a compact subinterval $J(\ni 0)$ of ${\mathbb{R}}$ and $C_2>0$ such that
\begin{equation*}
\frac{Q''(x)}{|Q'(x)|}\geqslant C_2\frac{|Q'(x)|}{Q(x)},  \quad a.e. \quad  x\in {\mathbb{R}}\setminus J,
\end{equation*}
then we say that $w=\exp(-Q)$ is in the class $\mathcal{F}(C^2+)$.
If $T(x)$ is bounded,
then $w$ is called a Freud-type weight, and if $T(x)$ is unbounded, then $w$ is Erd\"os-type weight.
\end{definition}

Some typical examples in $\mathcal{F}(C^2+)$ are given as follows:
\begin{example}[see \cite{[16]}]\label{Examples1.2}{\rm
(1)
For $\alpha>1$ and a non-negative integer  $\ell$, we put
\begin{equation*}
Q(x)=Q_{\ell,\alpha}(x):=\exp_{\ell}(|x|^{\alpha})-\exp_{\ell}(0),
\end{equation*}
where for $\ell \geqslant 1$,
\begin{equation*}
\exp_{\ell}(x):=\exp(\exp(\exp(\cdots\exp x)\ldots))  \quad (\ell \textrm{-times})
\end{equation*}
and $\exp_0(x):=x$.
\item[\,\,\,(2)]
For $m\geqslant 0$, $\alpha\geqslant 0$ with $\alpha+m>1$, we put
\begin{equation}\label{eq1.4}
Q(x)=Q_{\ell,\alpha,m}(x):=|x|^m\{\exp_{\ell} (|x|^\alpha) -\alpha^*\exp_{\ell}(0)\},
\end{equation}
where for $\ell>0$ we suppose $\alpha^*=0$ if $\alpha=0$; $\alpha^*=1$ otherwise. For $\ell =0$ we suppose $m>1$ and $\alpha=0$.
Note that $Q_{\ell,0,m}$ is a Freud-type weight.
\item[\,\,\,(3)]
For $\alpha>1$, we put
\begin{equation*}
Q(x)=Q_{\alpha}(x):=(1+|x|)^{|x|^\alpha} -1.
\end{equation*}
}
\end{example}
To consider the higher order Hermite-Fej\'er interpolation polynomial
we define a class of further strengthened  weights for $\nu\geqslant 2$
than Definition \ref{Definition1.1} as follows:
\begin{definition}[{cf.\cite{[7]}}]\label{Definition1.3}
Let $w(x)=\exp(-Q(x))\in \mathcal{F}(C^2+)$, and let $\nu\geqslant 2$ be an integer.
Assume that $Q(x)$ is a $\nu$-times continuously differentiable function on ${\mathbb{R}}$ and satisfies the following:
\item[\,\,\,(a)] $Q^{(\nu+1)}(x)$ exists and $Q^{(i)}(x),  0\le i\le \nu+1,$ are positive for $x>0$.
\item[\,\,\,(b)] There exist constants $C_i>0$ such that
\begin{equation*}
\left|Q^{(i+1)}(x)\right|\le C_i\left|Q^{(i)}(x)\right|\frac{|Q'(x)|}{Q(x)},
\quad x\in {\mathbb{R}}\backslash\{0\}, \quad i=1,2,\ldots \nu.
\end{equation*}
\item[\,\,\,(c)] There exist $0\le\delta<1$ and $c_1>0$ such that
\begin{equation}\label{eq1.5}
Q^{(\nu+1)}(x)\le C\left(\frac{1}{x}\right)^{\delta},   \quad  x\in (0,c_1].
\end{equation}
Then we say that  $w(x)=\exp(-Q(x))$ is in the class $\mathcal{F}_\nu(C^2+)$.
\item[\,\,\,(d)] Suppose one of the following:
\begin{enumerate}
\item[(d-1)] $Q'(x)/Q(x)$  is quasi-increasing on a certain positive interval $[c_2,\infty)$.
\item[(d-2)] $Q^{(\nu+1)}(x)$  is non-decreasing on a certain positive interval   $[c_2,\infty)$.
\item[(d-3)] There exist constants $C>0$ and $0\le \delta<1$
such that $Q^{(\nu+1)}(x)\le C(1/x)^{\delta}$  on $(0,\infty)$.
\end{enumerate}
Then we write $w(x)=\exp(-Q(x))\in\mathcal{\tilde{F}}_\nu(C^2+)$.
\end{definition}
In this paper we treat the weight $w_\rho(x)=|x|^{\rho}\exp(-Q(x))$, $\rho\geqslant 0$ of type (\ref{eq1.1}).
For $w_\rho(x)=|x|^{\rho}\exp(-Q(x))$, $w(x)=\exp(-Q(x))\in \tilde{\mathcal{F}}_\nu(C^2+)$,
we write $w_\rho\in \tilde{\mathcal{F}}_{\nu,\rho}(C^2+)$.

In the following, we give specific examples of $w \in \tilde{\mathcal{F}}_\nu(C^2+)$.
\begin{example}[{cf.\cite[Theorem 3.1]{[6]}}]\label{Example1.4} {\rm
Let $\nu$ be a positive integer,
and let $Q_{\ell,\alpha,m}$ be defined in (\ref{eq1.4}).
\item[\,\,\,(1)] Let $m$ and $\alpha$ be non-negative even integers with $m+\alpha>1$.
Then $w(x)=\exp(-Q_{\ell,\alpha,m})\in \mathcal{F}_\nu(C^2+)$, and one has the following.
\item[\qquad (a)] If $\ell>0$, then we see that $Q'_{\ell,\alpha,m}(x)/Q_{\ell,\alpha,m}(x)$ is quasi-increasing
on a certain positive interval $(c_1,\infty)$, and $Q_{\ell,0,m}(x)$ is non-decreasing on $(0,\infty)$.
\item[\qquad (b)] If $\ell=0$, then we see that $Q_{0,0,m}(x),  m\geqslant 2$, is non-decreasing on $(0,\infty)$.
Hence $w(x)=\exp(-Q_{\ell,\alpha,m})\in \tilde{\mathcal{F}}_{\nu}(C^2+)$.
\item[\,\,\,(2)] Let $m+\alpha-\nu>0$. Then $w(x)=\exp(-Q_{\ell,\alpha,m})\in \mathcal{F}_\nu(C^2+)$,
and one has the following.
\item[\qquad (c)] If $\ell\geqslant 2$ and $\alpha>0$, then there exists a constant $c_1>0$
such that $Q'_{\ell,\alpha,m}(x)/Q_{\ell,\alpha,m}(x)$ is quasi-increasing on $(c_1,\infty)$.
\item[\qquad (d)] Let $\ell=1$.
If $\alpha\geqslant 1$, then there exists a constant $c_2>0$
such that $Q'_{1,\alpha,m}(x)/Q_{1,\alpha,m}(x)$ is quasi-increasing on $(c_2,\infty)$,
and if $0<\alpha<1$,
then $Q'_{1,\alpha,m}(x)/Q_{1,\alpha,m}(x)$ is quasi-decreasing on $(c_2,\infty)$.
\item[\qquad (e)] Let $\ell=1$, and $0<\alpha<1$, then $Q^{(\nu+1)}_{1,\alpha,m}(x)$ is non-decreasing
on a certain positive interval on $(c_2,\infty)$.
Hence $w(x)=\exp(-Q_{\ell,\alpha,m})\in \tilde{\mathcal{F}}_{\nu}(C^2+)$.
}
\end{example}
\begin{example}[{\cite[Theorem 3.5]{[6]}}]\label{Example1.5} {\rm
Let $\nu$ be a positive integer and $\alpha > \nu$. Then $w(x)=\exp(-Q_{\alpha}(x))$ belongs to
$\mathcal{F}_{\nu}(C^2+)$.
Moreover, there exists a positive constant $c_2 > 0$ such that
$Q_{\alpha}'(x)/Q_{\alpha}(x)$ is quasi-increasing on $(c_2, \infty)$.
Hence $w(x)=\exp(-Q_{\alpha})\in \tilde{\mathcal{F}}_{\nu}(C^2+)$.
}
\end{example}

In Section 2, we report the $L_p$-convergence theorems.
We write some lemmas to prove the theorems and then we prove them in Section 3.

In what follows we abbreviate several notations as
$x_{k,n}:=x_{k,n,\rho},  h_{kn}(x):=h_{k,n,\rho}(\nu,x)$,
$l_{kn}(x):=l_{k,n,\rho}(x)$,
$h_{skn}(x):=h_{s,k,n,\rho}(\nu,x)$ and $p_n(x):=p_{n,\rho}(x)$ if there is no confusion.
For arbitrary nonzero real valued functions $f(x)$ and $g(x)$, we write $f(x)\sim g(x)$
if there exist constants $C_1,  C_2>0$ independent of $x$ such that $C_1 g(x)\le f(x)\le C_2g(x)$ for all $x$.
For arbitrary positive sequences $\{c_n\}_{n=1}^\infty$ and $\{d_n\}_{=1}^\infty$ we define $c_n\sim d_n$ similarly.

Throughout this paper $C,C_1,C_2,\ldots$ denote positive constants independent of $n,x,t$ or polynomials $P_n(x)$,
and the same symbol does not necessarily denote the same constant in different occurrences.

\setcounter{equation}{0}
\section{Theorems}

We use the following notations.\\
(1)  Mhaskar-Rakhmanov-Saff  numbers $a_x$:
\begin{equation*}
x=\frac{2}{\pi}\int_0^1\frac{a_xuQ'(a_xu)}{(1-u^2)^{1/2}}du, \quad x>0.
\end{equation*}
(2)
\begin{equation}\label{eq2.1}
\varphi_u(x)=
\left\{
\begin{array}{lr}
\frac{|x|}{u}\frac{1-\frac{|x|}{a_{2u}}}{\sqrt{1-\frac{|x|}{a_u}+\delta_u}},  & |x|\le a_u;  \\
\varphi_u(a_u),          & a_u<|x|,
\end{array}
         \right.
\end{equation}
where
\[
\delta_u=(uT(a_u))^{-2/3},  \quad u>0.
\]
In the rest of this paper, we assume the following:
\begin{assumption}\label{Assumption2.1}{\rm
Let the weight $w_\rho\in \tilde{\mathcal{F}}_{\nu,\rho}(C^2+)$, $\rho\geqslant 0$. \\
(a) If $T(x)$ is bounded, then we suppose that for some $C>0$
\begin{equation*}
  Q(x)\geqslant C|x|^{2},
\end{equation*}
and for $\delta$ in (\ref{eq1.5}),
\begin{equation}\label{eq2.2}
  a_n\le C n^{1/(1+\nu-\delta)}.
\end{equation}
(b) There exist $0\le \gamma< 1$ and $C(\gamma)>0$ such that
\begin{equation}\label{eq2.3}
   T(a_n)\le C(\gamma)n^\gamma,
\end{equation}
here, if $T(x)$ is bounded, that is, the weight $w$ is a Freud weight, then we set $\gamma=0$. We put
\begin{equation}\label{eq2.4}
  \varepsilon_n:=
\begin{cases}
         \frac{a_n}{n},& \frac{T(a_n)}{a_n}<1; \\
         \frac{1}{n^{1-\gamma}},& \frac{T(a_n)}{a_n}\geqslant 1.
\end{cases}
\end{equation}
}
\end{assumption}
\begin{lemma}[{\cite[Theorem 1.4]{[4]}}]\label{Lemma2.2}
{\rm (1)} Let $w=\exp(-Q)\in \mathcal{F}(C^2)$, and let $T(x)$ be unbounded.
Then, for any $\eta>0$ there exists $C(\eta)>0$ such that for $t\geqslant 1$,
\begin{equation}\label{eq2.5}
  a_t\le C(\eta)t^{\eta}.
\end{equation}
\item[\,\,\,(2)] Let $\lambda:=C_1$ be the constant in Definition \ref{Definition1.1} (e), that is,
\begin{equation*}
  \frac{Q''(x)}{Q'(x)}\le \lambda\frac{Q'(x)}{Q(x)},  \quad a.e.\quad  x\in \mathbb{R}\backslash\left\{0\right\}.
\end{equation*}
If $1<\lambda$, and if $\eta$ is defined in (\ref{eq2.5}), then there exists $C(\lambda,\eta)$ such that
\begin{equation*}
  T(a_t)\le  C(\lambda,\eta)t^{\frac{2(\eta+\lambda-1)}{\lambda+1}},
\end{equation*}
and if $\lambda\le 1+\eta$, then there exists $C(\lambda,\eta)$ such that
\begin{equation}\label{eq2.6}
  T(a_t)\le C(\lambda,\eta)t^{4\eta}, \quad  t\geqslant 1.
\end{equation}
\end{lemma}
\begin{remark}\label{Remark2.3} {\rm
(1) If $T(x)$ is unbounded, then for any $L>0$ we have $Q(x) \ge C|x|^L$(\cite[Lemma 3.2]{[24]})
and (\ref{eq2.2}) holds (\cite[Theorem 1.4]{[4]}).
\item[\,\,\,(2)]  (\ref{eq2.3}) holds for
\begin{equation*}
  \gamma=\frac{2(\eta+\lambda-1)}{\lambda+1}, \quad 0<\lambda<3.
\end{equation*}
\item[\,\,\,(3)]
Let us denote the MRS-number $a_n$ and the function $T(x)$
for $Q_{r,\alpha,m}$ by $a_{n,r,\alpha,m}$ and $T_{r,\alpha,m}$ respectively.
Then $T_{r,\alpha,m}(a_{n,r,\alpha,m})$ satisfies (\ref{eq2.3}).
In fact, we obtain from \cite[Proposition 4]{[23]} and \cite[Example 2 and (1.34)]{[16]},
\begin{equation*}
  a_{n,r,\alpha,m}\sim (1+\log_r^+ n)^{1/\alpha},
\end{equation*}
where
\begin{equation*}
\log_r^+ (x)=
\left\{
    \begin{array}{ll}
    \log(\log(\log(...\log x)))\,\,  (r \textrm{times}),& \textrm{ if } x>\exp_r(0),\\
    0,& \textrm{ otherwise },
    \end{array}
\right.
\end{equation*}
and for every $0< \gamma< 1$,
\begin{equation*}
  T_{r,\alpha,m}(a_{n,r,\alpha,m})\sim (1+\log_r^+n)(1+\log_{r-1}^+n)(1+\log^+n)<n^\gamma
\end{equation*}
by
\begin{equation*}
  T_{r,\alpha,m}(a_{n,r,\alpha,m})=m+T_{r,\alpha,0}(a_{n,r,\alpha,m}).
\end{equation*}
\item[\,\,\,(4)] If the weight $w=\exp(-Q)\in\mathcal{F}(C^2+)$ satisfies
\begin{equation*}
\mu=\lim\inf_{x\rightarrow \infty}\frac{Q(x)Q''(x)}{(Q'(x))^2}
=\lim\sup_{x\rightarrow \infty}\frac{Q(x)Q''(x)}{(Q'(x))^2},
\end{equation*}
then we say that the weight $w$ is regular. For the regular weight $w$ we have
\begin{equation*}
  \mu=\lim_{x\rightarrow \infty}\frac{Q(x)Q''(x)}{(Q'(x))^2}=1,
\end{equation*}
therefore, we obtain (\ref{eq2.6}) for any fixed $\eta>0$ (see \cite[Corollary 5.5]{[24]}).
We note that all of examples in Example \ref{Example1.4} are regular.
\item[\,\,\,(5)] (\ref{eq2.5}) means
\begin{equation*}
0<C\le\frac{n}{a_n}\left(\frac{T(a_n)}{a_n}\right)^{\nu-1}.
\end{equation*}
}
\end{remark}
Let
\begin{eqnarray*}
X_n(\nu,f;x) &:=&\sum_{k=1}^nf(x_{k,n})l_{kn}^{\nu}(x)
\sum_{i=0}^{\nu -2}e_i(\nu,k,n)(x-x_{k,n})^i,\\
Y_n(\nu,f;x)&:=&\sum_{k=1}^nf(x_{k,n})l_{kn}^{\nu}(x)e_{\nu-1}(\nu,k,n)(x-x_{k,n})^{\nu-1},\\
Z_n(l,\nu,f;x)&:=&\sum_{k=1}^n    \sum_{s=1}^lf^{(s)}(x_{k,n})l_{kn}^{\nu}(x)
\sum_{i=s}^{\nu-1}e_{si}(\nu,k,n)(x-x_{k,n})^i.
\end{eqnarray*}
Then we know
\begin{equation*}
  L_n(\nu,f;x)=X_n(\nu,f;x)+Y_n(\nu,f;x),
\end{equation*}
and
\begin{equation*}
  L_n(l,\nu,f;x)=L_n(\nu,f;x)+Z_n(l,\nu,f;x).
\end{equation*}
Let
\begin{equation}\label{eq2.7}
  \Phi(x):=\frac{1}{(1+Q(x))^{2/3}T(x)}.
\end{equation}
Then we see that for $0<d\le |x|$,
\[
\Phi(x)\sim  \frac{Q(x)^{\frac{1}{3}}}{xQ'(x)}.
\]
Moreover, if we define
\begin{equation}\label{eq2.8}
  \Phi_n(x):=\max\left\{\delta_n, 1-\frac{|x|}{a_n}\right\}, \quad n=1,2,3,....,
\end{equation}
then we have the following:
\begin{lemma}[{\cite[Lemma 3.4]{[7]}}]\label{Lemma2.4}
For $x\in \mathbb{R}$ we have
\begin{equation*}
  \Phi(x)\le C\Phi_n(x), n\geqslant 1.
\end{equation*}
\end{lemma}
We have a chain of results as follows.
In this section, we let $w_\rho\in \mathcal{\tilde{F}}_{\nu,\rho}(C^2+)$, $\rho\geqslant 0$,
and we suppose Assumption \ref{Assumption2.1}.
In addition, we let $\alpha>0, \Delta>-1$, $1<p<\infty$.
Let $C_f$ be a positive constant depending only on $f$, and let $C>0$ be a constant.
\begin{proposition} \label{Proposition 2.5}
Let $\nu=2,3,4,...$, and let
\begin{equation}\label{eq2.9}
  \Delta\geqslant \frac{1}{p}-\min\left\{1,\alpha\right\}.
\end{equation}
For $f\in C(\mathbb{R})$ satisfying
\begin{equation} \label{eq2.10}
|f(x)|(1+|x|)^\alpha \left\{\Phi^{-\frac{3}{4}}(x)w_\rho(x)\right\}^{\nu}\le C_f, \quad x\in \mathbb{R},
\end{equation}
we have
\begin{equation}\label{eq2.11}
  \left\|(1+|x|)^{-\Delta}\left\{\Phi^{\frac{3}{4}}(x)w(x)\left(|x|+\frac{a_n}{n}\right)^{\rho}\right\}^{\nu}
  X_n(\nu,f;x)\right\|_{L_p(\mathbb{R})}
\le C C_f.
\end{equation}
\end{proposition}
\begin{proposition}\label{Proposition2.6}
Let $\nu=2,4,6,...$ and assume that (\ref{eq2.9}) holds.
For $f\in C(\mathbb{R})$ satisfying
\begin{equation}\label{eq2.12}
|f(x)|(1+|x|)^\alpha\left\{\Phi^{-\frac{3}{4}}(x)w_\rho(x)\right\}^{\nu}
\left(|Q'(x)|+\frac{1}{|x|}\right)\le C_f, \quad  x\in \mathbb{R}\setminus \{0\}
\end{equation}
we have
\begin{equation*}
  \left\|(1+|x|)^{-\Delta}\left\{\Phi^{\frac{3}{4}}(x)w(x)\left(|x|+\frac{a_n}{n}\right)^{\rho}\right\}^{\nu}
  Y_n(\nu,f;x)\right\|_{L_p(\mathbb{R})}
\le CC_f\varepsilon_n(a_n+\log n),
\end{equation*}
where $\varepsilon_n$ is defined by (\ref{eq2.4}).
\end{proposition}
\begin{proposition}\label{Proposition2.7}
Let $\nu=2,3,4,...$ and assume that (\ref{eq2.9}) holds.
For $f\in C^l(\mathbb{R})$, $0\le l\le \nu-1$ satisfying
\begin{equation}\label{eq2.13}
|f^{(s)}(x)|(1+|x|)^\alpha \left\{\Phi^{-\frac{3}{4}}(x)w_\rho(x)\right\}^{\nu}\le C_f,
\quad x\in \mathbb{R}, \,\,\, s=1,2,...,l,
\end{equation}
we have
\begin{eqnarray*}
&&\left\|(1+|x|)^{-\Delta}\left\{\Phi^{\frac{3}{4}}(x)w(x)\left(|x|+\frac{a_n}{n}\right)^{\rho}\right\}^{\nu}
Z_n(l,\nu,f;x)\right\|_{L_p(\mathbb{R})}\\
&\le& C_f\frac{a_n^2\log n}{n}
\left\{
\begin{array}{lr}
         1, & \Delta p<1; \\
         \frac{\log a_n}{a_n},& \Delta p=1;\\
         \frac{1}{a_n}, &\Delta p>1.
\end{array}
         \right.
\end{eqnarray*}
\end{proposition}
\begin{proposition}\label{Proposition2.8}
Let $\nu=2,3,4,...$ and assume that (\ref{eq2.9}) holds. Let $P\in \mathcal{P}_{\nu n-1}$ be fixed.
Then, we have
\begin{equation*}
\left\|(1+|x|)^{-\Delta}
\left\{\Phi^{\frac{3}{4}}(x)w(x)\left(|x|+\frac{a_n}{n}\right)^{\rho}\right\}^{\nu}
\left(L_n(\nu,P;x)-P(x)\right)
\right\|_{L_p(\mathbb{R})}
\rightarrow 0 \quad \textrm{as} \quad  n\rightarrow \infty.
\end{equation*}
\end{proposition}
\begin{proposition}\label{Proposition2.9}
Let $\nu=2,4,6,...$ and assume that (\ref{eq2.9}) holds. Let $P\in \mathcal{P}_{\nu n-1}$ be fixed. Then, we have
\begin{equation*}
\left\|(1+|x|)^{-\Delta}\left\{\Phi^{\frac{3}{4}}(x)w(x)
\left(|x|+\frac{a_n}{n}\right)^{\rho}\right\}^{\nu}(X_n(\nu,P;x)-P(x))
\right\|_{L_p(\mathbb{R})}
\rightarrow 0 \quad \textrm{as} \quad  n\rightarrow \infty.
\end{equation*}
\end{proposition}
\begin{remark}\label{Remark2.10} {\rm
Let $f\in C(\mathbb{R})$ satisfy that for given $0<\eta<1$ and $\alpha\geqslant 0$,
\begin{equation}\label{eq2.14}
|f(x)|(1+|x|)^\alpha w^{\nu-\eta}(x)\left\{|Q'(x)|+\frac{1}{|x|}\right\} \le C_f,  \quad x\in\mathbb{R}
\quad \textrm{ and } \quad
\lim_{x\rightarrow 0}\frac{f(x)}{x}\le C_f,
\end{equation}
where $C_f$ is a constant depending only on $f$. Then,
\item[\,\,\,(1)] (\ref{eq2.10}) and (\ref{eq2.12}) hold.
\item[\,\,\,(2)] Let $f\in C^l(\mathbb{R})$ for a certain $0\le l\le \nu-1$.
Then (\ref{eq2.13}) holds.
\item[\,\,\,(3)]
Let (\ref{eq2.14}) be satisfied, then we see $f(0)=0$.
}
\end{remark}
\begin{theorem}\label{Theorem2.11}
Let $\nu=2,4,6,...$ and assume that (\ref{eq2.9}) holds. For $f(x)$ satisfying (\ref{eq2.14}) we have
\begin{equation*}
\left\|(1+|x|)^{-\Delta}\left\{\Phi^{\frac{3}{4}}(x)w(x)\left(|x|+\frac{a_n}{n}\right)^{\rho}\right\}^{\nu}
(L_n(\nu,f;x)-f(x))\right\|_{L_p(\mathbb{R})}
\rightarrow 0 \quad \textrm{as} \quad  n\rightarrow \infty.
\end{equation*}
\end{theorem}
\begin{theorem}\label{Theorem2.12}
Let $\nu=2,4,6,...$ and assume that (\ref{eq2.9}) holds. For $f(x)$ satisfying (\ref{eq2.14}) and (\ref{eq2.13}), we have
\begin{equation*}
\left\|(1+|x|)^{-\Delta}\left\{\Phi^{\frac{3}{4}}(x)w(x)\left(|x|+\frac{a_n}{n}\right)^{\rho}\right\}^{\nu}
(L_n(l,\nu,f;x)-f(x))\right\|_{L_p(\mathbb{R})}
\rightarrow 0  \quad \textrm{as} \quad  n\rightarrow \infty.
\end{equation*}
\end{theorem}
\begin{remark}\label{Remark2.13} {\rm
From the formulas (\ref{eq2.2}), (\ref{eq2.4}) and (\ref{eq2.5}), we have
\begin{equation*}
\varepsilon_n(a_n+\log n) \to 0 \quad \textrm{ as } \quad n \to \infty,
\end{equation*}
and
\begin{equation*}
\frac{a_n^2\log n}{n}
\begin{cases}
         1, & \Delta p<1; \\
         \frac{\log a_n}{a_n},& \Delta p=1;\\
         \frac{1}{a_n}, &\Delta p>1
\end{cases}
\to 0 \quad \textrm{ as } \quad n \to \infty.
\end{equation*}
}
\end{remark}
\setcounter{equation}{0}
\section{Lemmas and Proof of Theorems}

 For the coefficients $e_{si}(\nu,k,n) (e_i(\nu,k,n):=e_{0i}(\nu,k,n))$
in (\ref{eq1.2}) or (\ref{eq1.3}) we have the following estimates.
\begin{lemma}[{\cite[Theorem 2.6]{[3]}}]\label{Lemma3.1}
Let $w(x)=\exp(-Q(x))\in \mathcal{F}(C^2+)$. We have the following.
For each $s=0,1,...,\nu-1$ and $i=s, s+1,...,\nu-1$,
\begin{equation*}
e_0(\nu,k,n)=1, \quad  |e_{si}(\nu,k,n)|\le C\left\{\frac{n}{(a_{2n}^2-x_{k,n}^2)^{1/2}}\right\}^{i-s}.
\end{equation*}
\end{lemma}
\begin{lemma}[{\cite[Theorem 4.4 and Lemma 3.7 (3.20)]{[5]}}]\label{Lemma3.2}
Let $w_\rho\in \tilde{\mathcal{F}}_{\nu,\rho}(C^2+)$.
If $x_{kn}\neq 0$ and $|x_{k,n}|\le a_n(1+\delta_n)$,
then $e_0(\nu,k,n)=1$ and for $i=1,2,\ldots, \nu-1$,
\begin{equation*}
|e_i(\nu,k,n)|\le C\left\{\frac{T(a_n)}{a_n}+|Q'(x_{k,n})|+\frac{1}{x_{k,n}}\right\}^{\langle i \rangle}
\left\{\frac{n}{a_{2n}-|x_{k,n}|}+\frac{T(a_n)}{a_n}\right\}^{i-{\langle i \rangle}},
\end{equation*}
where
\begin{equation*}
<i>=
\begin{cases}
    1,& \textrm{ if } i \textrm{ is odd},\\
    0,& \textrm{ if } i \textrm{ is even}.
    \end{cases}
\end{equation*}
For $x_{kn}=0$, we see $e_0(\nu,k,n)=1$ and
\begin{equation*}
|e_i(\nu,k,n)|\le C\left(\frac{n}{a_n}\right)^i \quad i=1,2,\ldots, \nu-1.
\end{equation*}
\end{lemma}
 We have the following.
\begin{lemma}[{\cite[Theorem 2.4]{[1]}}]\label{Lemma3.3}
Let $w_\rho\in \tilde{\mathcal{F}}_{\nu,\rho}(C^2+)$ $(\rho\geqslant 0)$, $0<p\le \infty$,
and $\beta\in \mathbb{R}$. Then given $r>1$, there exist $C, n_0, \alpha>0$
such that for $n\geqslant n_0$ and $P\in \mathcal{P}_n$,
\begin{equation*}
\left\|(Pw)(x)|x|^\beta\right\|_{L_p(a_{rn}\le |x|)}
\le \exp(-Cn^\alpha)\left\|(Pw)(x)|x|^\beta\right\|_{L_p(L\frac{a_n}{n}\le |x|\le a_n(1-L\eta_n))}.
\end{equation*}
\end{lemma}
\begin{lemma}\label{Lemma3.4} {\rm (1)}
Let $P\in \mathcal{P}_{\nu n -1}$. Then, we have
\begin{eqnarray*}
&&\left\|(1+|x|)^{-\Delta}\left\{\Phi^{\frac{3}{4}}(x)w(x)\left(|x|+\frac{a_n}{n}\right)^{\rho}\right\}^{\nu} P(x)
\right\|_{L_p(a_{2n}\le |x|)}\\
&\le& C e^{-C_1n^{\eta}}\left\|(1+|x|)^{-\Delta}\left\{\Phi^{\frac{3}{4}}(x)w(x)
\left(|x|+\frac{a_n}{n}\right)^{\rho}\right\}^{\nu}
P(x)\right\|_{L_p(|x|\le a_{2n})}.
\end{eqnarray*}
So, if
\begin{equation*}
  \left\|(1+|x|)^{-\Delta}\left\{\Phi^{\frac{3}{4}}(x)w(x)(|x|+1)^{\rho}\right\}^{\nu} P(x)
  \right\|_{L_p(\mathbb{R})}<\infty,
\end{equation*}
then we have
\begin{equation*}
\left\|(1+|x|)^{-\Delta}\left\{\Phi^{\frac{3}{4}}(x)w(x)\left(|x|+\frac{a_n}{n}\right)^{\rho}\right\}^{\nu} P(x)
\right\|_{L_p(a_{2n}\le |x|)}\rightarrow 0
\quad \textrm{as} \quad  n\rightarrow \infty.
\end{equation*}
\end{lemma}
\begin{proof}
(1) Now, we denote the Mhaskar-Rakhmanov-Saff numbers
for the exponents $Q(x)$ and $\nu Q(x)$ by $a_n(Q)$ and $a_n(\nu Q)$ respectively.
Then, we have $a_n(Q)=a_{\nu n}(\nu Q)$.
From (\ref{eq2.7})
we note that $(1+|x|)^{\Delta}\Phi(x)^{-3\nu/4}$ is quasi-increasing.
Then applying Lemma \ref{Lemma3.3} with $w^{\nu}_{\nu\rho}(x)=\exp(-\nu Q(x))|x|^{\nu\rho}$, we have
\begin{eqnarray*}
&&\left\|(1+|x|)^{-\Delta}\left\{\Phi^{\frac{3}{4}}(x)
w(x)\left(|x|+\frac{a_n}{n}\right)^{\rho}\right\}^{\nu} P(x)\right\|_{L_p(a_{2n}(Q)\le |x|)}\\
&\le&
C\left\|w^{\nu}(x)\left(|x|+\frac{a_n}{n}\right)^{\nu\rho} P(x)\right\|_{L_p(a_{2\nu n}(\nu Q)\le |x|)}\\
&\le& C e^{-C(\nu n)^{\eta}}\left\|w^{\nu}(x)\left(|x|+\frac{a_n}{n}\right)^{\nu\rho} P(x)\right\|_{L_p(|x|\le a_{2\nu n}(\nu Q))}\\
&\le& C \frac{e^{-C\nu^{\eta} n^{\eta}}}{(1+a_{2n}(Q))^{-\Delta}\Phi(a_{2n}(Q))^{3\nu/4}}\\
&& \times\left\|(1+|x|)^{-\Delta}\left\{\Phi^{\frac{3}{4}}(x)(w(x)\left(|x|+\frac{a_n}{n}\right)^{\rho}\right\}^{\nu}
P(x)\right\|_{L_p(|x|\le a_{2n}(Q))}\\
&\le& C e^{-C_1n^{\eta}}
\left\|(1+|x|)^{-\Delta}\left\{\Phi^{\frac{3}{4}}(x)w(x)\left(|x|+\frac{a_n}{n}\right)^{\rho}\right\}^{\nu} P(x)
\right\|_{L_p(|x|\le a_{2n}(Q))}\\
&&\rightarrow 0  \quad \textrm{as} \quad  n\rightarrow \infty,
\end{eqnarray*}
because that $(1+a_{2n}(Q))^{\Delta}\Phi(a_{2n}(Q))^{-3\nu/4}$ has order $n^s$ for some $s>0$.
\end{proof}
In the rest of this section
we let $w_\rho\in \mathcal{\tilde{F}}_{\nu,\rho}(C^2+)$, $\rho\geqslant 0$,
and we suppose Assumption \ref{Assumption2.1}.
In addition, we let $\alpha>0, \Delta>-1$, $1<p<\infty$, and $C_f$ be a positive constant depending only on $f$.

To simplify the proofs of theorems we use the results of \cite{[7]}.
   In what follows we use the following notation.
Let $x_{0,n}:=x_{1,n}+\varphi_n(x_{1,n})$, and $x_{n+1,n}=-x_{0,n}$. When $|x|\le x_{0,n}$, we define
\begin{equation}\label{eq3.1}
  x_{m,n}:=x_{m(x),n}; \quad  |x-x_{m,n}|=\min_{0\le j\le n}|x-x_{j,n}|.
\end{equation}
If $|x-x_{j+1,n}|=|x-x_{j,n}|$, then we set $x_{m,n}:=x_{j,n}$.
If $x_{0,n}<x$, then we put $m=0$. And if $x<x_{n+1,n}$,
then we put $m=n+1$. Here, we note that there exists a constant $\delta>0$
such that $|x-x_{m,n}|\le \delta\varphi_n(x_{m,n})$.

\begin{lemma}[cf.{\cite{[7]}}]\label{Lemma3.5}
Let $m:=m(x)$ be defined by (\ref{eq3.1}).
\item[\,\,\,(1)] If (\ref{eq2.10}) holds, then, we have
\begin{equation}\label{eq3.2}
\left\{\Phi^{\frac{3}{4}}(x)w(x)\left(|x|+\frac{a_n}{n}\right)^{\rho}\right\}^{\nu}|X_n(\nu,f;x)|
\le C_f\sum_{j=0}^n(1+|x_{j,n}|)^{-\alpha}\sum_{i=0}^{\nu-2}\left(\frac{1}{1+|m-j|}\right)^{\nu-i}.
\end{equation}
\item[\,\,\,(2)] If (\ref{eq2.12}) holds, then, we have
\begin{equation}\label{eq3.3}
\left\{\Phi^{\frac{3}{4}}(x)w(x)\left(|x|+\frac{a_n}{n}\right)^{\rho}\right\}^{\nu}|Y_n(\nu,f;x)|
\le C_f\varepsilon_n\sum_{j=0}^n(1+|x_{j,n}|)^{-\alpha}\left(\frac{1}{1+|m-j|}\right).
\end{equation}
\item[\,\,\,(3)] If (\ref{eq2.13}) holds, then, we have
\begin{eqnarray}\label{eq3.4} \nonumber
&& \left\{\Phi^{\frac{3}{4}}(x)w(x)\left(|x|+\frac{a_n}{n}\right)^{\rho}\right\}^{\nu}|Z_n(l,\nu,f;x)|\\
&\le& C_f\frac{a_n}{n}\sum_{j=0}^n(1+|x_{j,n}|)^{-\alpha}\sum_{i=0}^{\nu-1}\left(\frac{1}{1+|m-j|}\right)^{\nu-i}.
\end{eqnarray}
\end{lemma}
\begin{proof}
In \cite[Propositions 3.7, 3.8 and 3.9]{[7]}
we got the similar estimations in $L_{\infty}(\mathbb{R})$-space.
Now, in there we can exchange $f(x)$ with $f(x)(1+|x|)^{\alpha}$, and then we obtain the results in $L_{p}(\mathbb{R})$
by the similar methods as the proofs of \cite[Propositions 3.7, 3.8 and 3.9]{[7]}
(see \cite[p.16-p.23]{[7]}).
\end{proof}
\begin{lemma}\label{Lemma3.6}
{\rm (1)}  {\cite[Lemma 4.3]{[7]}} Uniformly for $n \ge 1$,
\begin{equation}\label{eq3.6}
\sup_{x\in \mathbb{R}}\left|p_{n,\rho}(x)w(x)\right|
\left(|x|+\frac{a_n}{n}\right)^{\rho}\Phi^{\frac{1}{4}}(x)
\le C a_n^{-\frac{1}{2}}.
\end{equation}
\item[\,\,\,(2)]\cite[Theorem 2.5(a)]{[2]}
Uniformly for $n \ge 1$,
\begin{equation}\label{eq3.5}
  |p'_{n,\rho}w|(x_{j,n})\left(|x_{j,n}|+\frac{a_n}{n}\right)^{\rho}\sim \varphi_n(x_{j,n})^{-1}[a^2_n-x^2_{j,n}]^{-1/4}.
\end{equation}
\end{lemma}
\begin{lemma}\label{Lemma3.7}
Let $w(x)=\exp(-Q(x))\in \mathcal{F}(C^2+)$. For the zeros $x_{j,n}=x_{j,n,\rho}$, we have the following.
\item[\,\,\,(1)] {\rm {\cite[Theorem 2.2, (b)]{[2]}}} For $n\geqslant 1$ and $1\le j\le n-1$,
\begin{equation*}
x_{j,n}-x_{j+1,n}\sim \varphi_{n}(x_{j,n}),
\end{equation*}
and {\rm \cite[Lemma A.1 (A.3)]{[2]}}
\begin{equation*}
 \varphi_{n}(x_{j,n})\sim \varphi_n(x_{j+1,n}).
\end{equation*}
\item[\,\,\,(2)] {\rm {\cite[Theorem 2.2, (a)]{[2]}}} For the minimum positive zero $x_{[n/2],n}$ ($[n/2]$ is the largest integer $\le n/2$), we have
\begin{equation*}
 x_{[n/2],n}\sim a_n n^{-1},
\end{equation*}
and for large enough $n$,
\begin{equation*}
 1-\frac{x_{1,n}}{a_{n}} \sim \delta_{n}.
\end{equation*}
\item[\,\,\,(3)]{\rm \cite[Lemma 4.7]{[2]}}
\begin{equation*}
b_n=\frac{\gamma_{n-1}}{\gamma_n} \sim a_n \sim x_{1,n}.
\end{equation*}
\end{lemma}
\begin{lemma}\label{Lemma3.8}
Let $w=\exp(-Q)\in \mathcal{F}(C^2+)$. Let $L>0$ be fixed. Then we have the following.
\item[\,\,\,(a)] {\rm (\cite[Lemma 3.5, (a)]{[16]})} Uniformly for $t>0$,
\begin{equation*}
a_{Lt}\sim a_t.
\end{equation*}
\item[\,\,\,(b)] {\rm (\cite[Lemma 3.5, (b)]{[16]})} Uniformly for $t>0$,
\begin{equation*}
Q^{(j)}(a_{Lt})\sim Q^{(j)}(a_t),  \quad j=0,1.
\end{equation*}
Moreover,
\begin{equation*}
T(a_{Lt})\sim T(a_t).
\end{equation*}
\item[\,\,\,(c)] {\rm (\cite[Lemma 3.11 (3.52)]{[16]})} Uniformly for $t>0$,
\begin{equation*}
\left|1-\frac{a_{Lt}}{a_t}\right|\sim \frac{1}{T(a_t)}.
\end{equation*}
\end{lemma}

   From Lemma \ref{Lemma3.6} it is sufficient that
   we estimate Proposition \ref{Proposition 2.5},  \ref{Proposition2.6} or \ref{Proposition2.7} for only $|x|\le a_{2n}$.
\begin{proof}[Proof of Proposition \ref{Proposition 2.5}]
First we set
\begin{eqnarray*}
X_n(\nu,f;x)&=&\sum_{j=1}^nf(x_{j,n})l_{jn}^{\nu}(x)\sum_{i=0}^{\nu -2}e_i(\nu,j,n)(x-x_{j,n})^i\\
&=&:\sum_{j\,\,;|x_{j,n}|\geqslant \frac{a_n}{3}}
+\sum_{j\,\,;|x_{j,n}|< \frac{a_n}{3}}=:X_{1,n}(\nu,f;x) +X_{2,n}(\nu,f;x).
\end{eqnarray*}
Using (\ref{eq3.2}), we see
\begin{eqnarray*}
&&\left|(1+|x|)^{-\Delta}\left\{\Phi^{\frac{3}{4}}(x)w(x)\left(|x|+\frac{a_n}{n}\right)^{\rho}\right\}^{\nu} X_{1,n}(\nu,f;x)\right|\\
&\le& C_f(1+|x|)^{-\Delta}\sum_{|x_{j,n}|\geqslant \frac{1}{3}a_n}(1+|x_{j,n}|)^{-\alpha}
\sum_{i=0}^{\nu-2}\left(\frac{1}{1+|m-j|}\right)^{\nu-i}\\
&\le& CC_fa_n^{-\alpha}(1+|x|)^{-\Delta}.
\end{eqnarray*}
Therefore, we have by (\ref{eq2.9})
\begin{eqnarray}\label{3.9} \nonumber
&&\left\|(1+|x|)^{-\Delta}\left\{\Phi^{\frac{3}{4}}(x)w(x)\left(|x|+\frac{a_n}{n}\right)^{\rho}\right\}^{\nu}
X_{1,n}(\nu,f;x)\right\|_{L_p(|x|\le a_{2n})}\\ \nonumber
&\le& C_fa_n^{-\alpha}\left\|(1+|x|)^{-\Delta}\right\|_{L_p(|x|\le a_{2n})}\\\nonumber
&\le& C_fa_n^{-\alpha}
\left\{
\begin{array}{ll}
         a_n^{\frac{1}{p}-\Delta}, & \Delta p<1, \\
         \log a_n, & \Delta p=1,\\
         1, & \Delta p>1
\end{array}
         \right.
\le C_f.
\end{eqnarray}
Now, we will estimate for $X_{2,n}(\nu,f,x)$.
For $|x|<1$, since  we have from (\ref{eq3.2})
\begin{equation*}
   \left\{\Phi^{\frac{3}{4}}(x)w(x)\left(|x|+\frac{a_n}{n}\right)^{\rho}\right\}^{\nu}|X_{2,n}(\nu,f;x)|\le C_f,
\end{equation*}
we see
\begin{equation}\label{eq3.7}
\left\|(1+|x|)^{-\Delta}\left\{\Phi^{\frac{3}{4}}(x)w(x)\left(|x|+\frac{a_n}{n}\right)^{\rho}\right\}^{\nu}
|X_{2,n}(\nu,f;x)|\right\|_{L_p(|x|\le 1)}\\
\le C_f.
\end{equation}
Let $|x| \ge 1$. We divide it into two sums as follows:
\begin{eqnarray*}
X_{2,n}(\nu,f;x) &=:&   \sum_{\substack{j;\,\,|x_{j,n}|< \frac{a_n}{3}, \\
|x-{x_{j,n}}|\geqslant \frac{|x|}{2}}}
+\sum_{\substack{j;\,\,|x_{jn}|< \frac{a_n}{3},\\ |x-{x_{j,n}}|< \frac{|x|}{2}}}
=: X_{2,n}^{[1]}(\nu,f;x) +X_{2,n}^{[2]}(\nu,f;x).
\end{eqnarray*}
Using (\ref{eq2.10}), Lemma \ref{Lemma3.7}(2), (\ref{eq3.6}) and (\ref{eq3.5}) with $\varphi_n(x_{j,n})\sim a_n/n$,
we see that
\begin{eqnarray}\label{eq3.8}
&&\left\{\Phi^{\frac{3}{4}}(x)w(x)\left(|x|+\frac{a_n}{n}\right)^{\rho}\right\}^{\nu}
\left|f(x_{j,n})l_{jn}^{\nu}(x)\right|\\  \nonumber
&\le& C_f
(1+|x_{j,n}|)^{-\alpha}
\left|\frac{\left\{\Phi^{\frac{3}{4}}(x_{j,n})\Phi^{\frac{3}{4}}(x)w(x)\left(|x|+\frac{a_n}{n}\right)^{\rho}\right\}
 p_n(x)}{p'_n(x_{j,n})w(x_{j,n})(\frac{a_n}{n}+|x_{j,n}|)^{\rho}}\right|^{\nu}\\ \nonumber
 &\le& CC_f (1+|x_{j,n}|)^{-\alpha}\left(\frac{a_n}{n}\right)^{\nu}
\end{eqnarray}
and Lemma \ref{Lemma3.1}
\begin{equation*}
\sum_{i=0}^{\nu -2}e_i(\nu,j,n)(x-x_{j,n})^i \le C\sum_{i=0}^{\nu -2}\left(\frac{n}{a_n}\right)^i|x|^{i-\nu}
\le C\left(\frac{n}{a_n}\right)^{\nu-2}|x|^{-2}.
\end{equation*}
Then, we have
\begin{eqnarray*}
&&\left|\left\{\Phi^{\frac{3}{4}}(x)w(x)\left(|x|+\frac{a_n}{n}\right)^{\rho}\right\}^{\nu}
X_{2,n}^{[1]}(\nu,f;x)\right|\\
&\le& C_f
\sum_{\substack{j\,\,;|x_{j,n}|< \frac{a_n}{3}, \\ |x-{x_{j,n}}|\geqslant \frac{|x|}{2}}}
(1+|x_{j,n}|)^{-\alpha}\left(\frac{a_n}{n}\right)^{2}
|x|^{-2}
\le CC_f\frac{a_n^2}{n}\frac{1}{x^2}.
\end{eqnarray*}
Therefore, we have, using $\Delta>-1$,
\begin{eqnarray}\label{eq3.9} \nonumber
&&\left\|(1+|x|)^{-\Delta}\left\{\Phi^{\frac{3}{4}}(x)w(x)\left(|x|+\frac{a_n}{n}\right)^{\rho}\right\}^{\nu}
X_{2,n}^{[1]}(\nu,f;x)\right\|_{L_p(1\le|x|\le a_{2n})}\\
&\le& C_f\frac{a_n^2}{n} \left\|(1+|x|)^{-(\Delta+2)}\right\|_{L_p(1\le|x|\le a_{2n})}
\le CC_f\frac{a_n^2}{n}\le CC_f
\end{eqnarray}
(note (\ref{eq2.2}) and (\ref{eq2.5})).
Since, for $|x-x_{j,n}| < \frac{|x|}{2}$,
we see  $|x_{j,n}|\sim |x|$,  the inequality (\ref{eq3.2}) implies
\begin{eqnarray}\label{eq3.10}
&&\left\{\Phi^{\frac{3}{4}}(x)w(x)\left(|x|+\frac{a_n}{n}\right)^{\rho}\right\}^{\nu}
X_{2,n}^{[2]}(\nu,f;x)\\  \nonumber
&\le& CC_f
\sum_{\substack{ j;\,\, |x_{j,n}|< \frac{a_n}{3},\\|x-{x_{j,n}}|< \frac{|x|}{2}}}(1+|x_{j,n}|)^{-\alpha}
\sum_{i=0}^{\nu-2}\left(\frac{1}{1+|m-j|}\right)^{\nu-i}
\le CC_f(1+|x|)^{-\alpha}.
\end{eqnarray}
Therefore, we obtain  by (\ref{eq2.9})
\begin{eqnarray}\label{eq3.11}
&&\quad \quad \left\|(1+|x|)^{-\Delta}\left\{\Phi^{\frac{3}{4}}(x)w(x)\left(|x|+\frac{a_n}{n}\right)^{\rho}\right\}^{\nu}
X_{2,n}^{[2]}(\nu,f;x)\right\|_{L_p(1\le |x|\le a_{2n})}\\ \nonumber
&\le& CC_f\left\|(1+|x|)^{-(\alpha+\Delta)}\right\|_{L_p(|x|\le a_{2n})}
\le CC_fa_n^{\frac{1}{p}-(\alpha+\Delta)}\le CC_f.
\end{eqnarray}
Hence, from (\ref{3.9}), (\ref{eq3.8}), (\ref{eq3.9}), (\ref{eq3.11}) and Lemma \ref{Lemma3.4} we conclude
(\ref{eq2.11})
\end{proof}
\begin{proof}[Proof of Proposition \ref{Proposition2.6}]
We repeat the methods of the proof of Proposition \ref{Proposition 2.5}.
\begin{equation*}
Y_n(\nu,f;x)=:\sum_{j;|x_{j,n}|\geqslant \frac{a_n}{3}}
+\sum_{j;|x_{j,n}|< \frac{a_n}{3}}=:Y_{1,n}(\nu,f;x) +Y_{2,n}(\nu,f;x).
\end{equation*}
Then we have by (\ref{eq3.3}), (\ref{eq2.9}), (\ref{eq2.2}) and (\ref{eq2.5})
\begin{eqnarray}\label{3.13} \nonumber
&&\left\|(1+|x|)^{-\Delta}\left\{\Phi^{\frac{3}{4}}(x)w(x)\left(|x|+\frac{a_n}{n}\right)^{\rho}\right\}^{\nu}
Y_{1,n}(\nu,f;x)\right\|_{L_p(|x|\le a_{2n})}\\\nonumber
&\le& C_fa_n^{-\alpha}\varepsilon_n\log n
\left\|(1+|x|)^{-\Delta}\right\|_{L_p(|x|\le a_{2n})}\\\nonumber
&\le& CC_fa_n^{-\alpha}\varepsilon_n\log n
\begin{cases}
         a_n^{\frac{1}{p}-\Delta},& \Delta p<1; \\
         \log a_n,& \Delta p=1;\\
         1, &\Delta p>1
\end{cases}\\  \nonumber
&\le& CC_f\varepsilon_n\log n
\begin{cases}
         1, & \Delta p<1; \\
         \frac{\log a_n}{a_n},& \Delta p=1;\\
         \frac{1}{a_n}, &\Delta p>1.
\end{cases}
\end{eqnarray}
Now, we estimate $Y_{2,n}(\nu,f;x)$. For $|x|\le 1$ we obtain that from (\ref{eq3.3}),
\begin{equation*}
\left\|(1+|x|)^{-\Delta}\left\{\Phi^{\frac{3}{4}}(x)w(x)\left(|x|+\frac{a_n}{n}\right)^{\rho}\right\}^{\nu} Y_{2,n}(\nu,f;x)
\right\|_{L_p(|x|\le 1)}
\le CC_f\varepsilon_n \log n.
\end{equation*}
Let $|x|\geqslant 1$ and
\begin{equation*}
Y_{2,n}(\nu,f;x)=:\sum_{\substack{j\,\,;|x_{j,n}|< \frac{a_n}{3},\\|x-{x_{j,n}}|\geqslant \frac{|x|}{2}>0}}
+\sum_{\substack{j\,\,;|x_{j,n}|< \frac{a_n}{3},\\|x-{x_{j,n}}|< \frac{|x|}{2}}}=:Y_{2,n}^{[1]}+Y_{2,n}^{[2]}.
\end{equation*}
Then we have similarly to  (\ref{eq3.10}) and (\ref{eq3.11}),
\begin{eqnarray}\nonumber
&&\left\|(1+|x|)^{-\Delta}\left\{\Phi^{\frac{3}{4}}(x)w(x)\left(|x|+\frac{a_n}{n}\right)^{\rho}\right\}^{\nu}
Y_{2,n}^{[2]}(\nu,f;x)\right\|_{L_p(1\le |x|\le a_{2n})}\\\nonumber
&\le& CC_f\varepsilon_n\log n \left\|(1+|x|)^{-(\alpha+\Delta)}\right\|_{L_p(|x|\le a_{2n})}
\le CC_fa_n^{\frac{1}{p}-(\alpha+\Delta)}\varepsilon_n\log n\le CC_f\varepsilon_n\log n
\end{eqnarray}
(see (\ref{eq2.9})).
Finally, we estimate $Y_{2,n}^{[1]}$.
Let $|x_{j,n}|< \frac{a_n}{3}$ and $|x-{x_{j,n}}|\geqslant \frac{|x|}{2}$.
Then, from Lemma \ref{Lemma3.2}, (\ref{eq2.12}), (\ref{eq3.6}) and (\ref{eq3.5}), we have
\begin{eqnarray}\label{eq3.12}
&&\left\{\Phi^{\frac{3}{4}}(x)w(x)\left(|x|+\frac{a_n}{n}\right)^{\rho}\right\}^{\nu}
\left|f(x_{j,n})l_{jn}^{\nu}(x)\right||x-x_{j,n}|^{\nu-1}e_{\nu-1}(\nu,k,n)\\  \nonumber
&\le& CC_f
(1+|x_{j,n}|)^{-\alpha}\left\{|Q'(x_{j,n})|+\frac{1}{|x_{j,n}|}\right\}^{-1}\\ \nonumber
&& \quad \times
\left|\frac{\left\{\Phi^{\frac{3}{4}}(x_{j,n})\Phi^{\frac{3}{4}}(x)w(x)\left(|x|+\frac{a_n}{n}\right)^{\rho}\right\}
 p_n(x)}{p'_n(x_{j,n})w(x_{j,n})(\frac{a_n}{n}+|x_{j,n}|)^{\rho}}\right|^{\nu} \frac{1}{|x-x_{j,n}|}\\ \nonumber
&& \quad \times \left\{\frac{T(a_n)}{a_n}+|Q'(x_{k,n})|+\frac{1}{x_{k,n}}\right\}
\left\{\frac{n}{a_{2n}-|x_{k,n}|}\right\}^{\nu-2} \\ \nonumber
&\le& CC_f (1+|x_{j,n}|)^{-\alpha}\left\{|Q'(x_{j,n})|+\frac{1}{|x_{j,n}|}\right\}^{-1}\\ \nonumber
&& \quad \times \Phi^{\frac{3\nu}{4}}(x_{j,n})\varphi_n^{\nu}(x_{j,n})\left(1-\frac{|x_{j,n}|}{a_n}\right)^{\frac{\nu}{4}}
\frac{1}{|x-x_{j,n}|}\\ \nonumber
&& \quad \times \left\{\frac{T(a_n)}{a_n}+|Q'(x_{k,n})|+\frac{1}{x_{k,n}}\right\}
\left\{\frac{n}{a_{2n}-|x_{k,n}|}\right\}^{\nu-2}.
\end{eqnarray}
On the other hand, from (\ref{eq2.1}), (\ref{eq2.8}) and $|x_{j,n}| < a_n/3$, we easily see
\begin{equation}\label{11}
\Phi_n^{\frac{3\nu}{4}}(x_{j,n})\varphi_n^{\nu}(x_{j,n})
\left(1-\frac{|x_{j,n}|}{a_n}\right)^{\frac{\nu}{4}}
\left(\frac{n}{a_{2n}-|x_{j,n}|}\right)^{\nu-2}
\le C\left(\frac{a_n}{n}\right)^2.
\end{equation}
Moreover, since $|Q'(x_{j,n})|+\frac{1}{|x_{j,n}|}\geqslant C>0$ for some constant $C$,
we see
\begin{equation}\label{eq3.14}
\left\{\frac{T(a_n)}{a_n}+|Q'(x_{j,n})|+\frac{1}{|x_{j,n}|}\right\}
\left\{|Q'(x_{j,n})|+\frac{1}{|x_{j,n}|}\right\}^{-1}
\le C\max\left\{1,\frac{T(a_n)}{a_n}\right\}
\end{equation}
and since $|x-{x_{j,n}}|\geqslant \frac{|x|}{2}$, we have
\begin{equation}\label{22}
\frac{1}{|x-x_{j,n}|} \le \frac{2}{|x|}.
\end{equation}
Thus, using (\ref{eq3.12}), (\ref{11}), (\ref{eq3.14}) and (\ref{22}),
we can estimate for $Y_{2,n}^{[1]}(\nu,f;x)$ as follows:
\begin{eqnarray*}
&&\left\{\Phi^{\frac{3}{4}}(x)w(x)\left(|x|+\frac{a_n}{n}\right)^{\rho}\right\}^{\nu} Y_{2,n}^{[1]}(\nu,f;x)\\
&\le& CC_f\max\left\{1,\frac{T(a_n)}{a_n}\right\}\frac{1}{|x|}\left(\frac{a_n}{n}\right)^2
\sum_{\substack{j;\,\,|x_{j,n}|< \frac{a_n}{3}, \\ |x-{x_{j,n}}|\geqslant \frac{|x|}{2}}}
(1+|x_{j,n}|)^{-\alpha} \\
&\le& CC_f\max\left\{1,\frac{T(a_n)}{a_n}\right\}\frac{1}{|x|}\frac{a_n^2}{n}
\le CC_f a_n \varepsilon_n\frac{1}{|x|}.
\end{eqnarray*}
Therefore we have by (\ref{eq2.9})
\begin{eqnarray*}
&&\left\|(1+|x|)^{-\Delta}\left\{\Phi^{\frac{3}{4}}(x)w(x)\left(|x|+\frac{a_n}{n}\right)^{\rho}\right\}^{\nu}
Y_{2,n}^{[1]}(\nu,f;x)\right\|_{L_p(1\le |x|\le a_{2n})}\\
&\le& CC_fa_n \varepsilon_n
\left\|(1+|x|)^{-(\Delta+1)}\right\|_{L_p(1\le |x|\le a_{2n})}
\le CC_f a_n \varepsilon_n a_n^{\frac{1}{p}-(\Delta+1)}
\le CC_fa_n\varepsilon_n.
\end{eqnarray*}
\end{proof}
\begin{proof}[Proof of Proposition \ref{Proposition2.7}]
Using (\ref{eq3.4}), we have from (\ref{eq2.2}) and (\ref{eq2.5})
\begin{eqnarray*}
&&\left\|(1+|x|)^{-\Delta}\left\{\Phi^{\frac{3}{4}}(x)w(x)\left(|x|+\frac{a_n}{n}\right)^{\rho}\right\}^{\nu}
Z_n(l,\nu,f;x)\right\|_{L_p(\mathbb{R})}\\
&\le& CC_f\frac{a_n\log n}{n}\left\|(1+|x|)^{-\Delta}\right\|_{L_p(|x|\le a_{2n})}\\
&\le& CC_f\frac{a_n\log n}{n}
\begin{cases}
         a_n^{\frac{1}{p}-\Delta},& \Delta p<1; \\
         \log a_n,& \Delta p=1;\\
         1, &\Delta p>1
\end{cases}
\le CC_f\frac{a_n^2\log n}{n}
\begin{cases}
         1, & \Delta p<1; \\
         \frac{\log a_n}{a_n},& \Delta p=1;\\
         \frac{1}{a_n}, &\Delta p>1.
\end{cases}
\end{eqnarray*}
\end{proof}
\begin{proof}[Proof of Proposition \ref{Proposition2.8}]
Since we see for a fixed polynomial $P\in \mathcal{P}_{\nu n-1}$
\begin{equation*}
L_n(\nu,P)-P=L_n(\nu,P)-L_n(\nu-1,\nu,P)=Z_n(\nu-1,\nu,P),
\end{equation*}
we can obtain from Proposition \ref{Proposition2.7},
\begin{eqnarray*}
&&\left\|(1+|x|)^{-\Delta}\left\{\Phi^{\frac{3}{4}}(x)w(x)\left(|x|+\frac{a_n}{n}\right)^{\rho}\right\}^{\nu}
\left\{L_n(\nu,P)-P\right\}\right\|_{L_p(\mathbb{R})}\\
&\le& CC_f\frac{a_n^2\log n}{n}
\begin{cases}
         1, & \Delta p<1; \\
         \frac{\log a_n}{a_n},& \Delta p=1;\\
         \frac{1}{a_n}, &\Delta p>1.
\end{cases}\rightarrow 0  \quad \textrm{ as } \quad  n\rightarrow \infty
\end{eqnarray*}
(note Remark \ref{Remark2.13}).
\end{proof}
\begin{proof}[Proof of Proposition \ref{Proposition2.9}]
Since we see for a fixed polynomial $P\in \mathcal{P}_{\nu n-1}$
\begin{equation*}
X_n(\nu,P)-P=L_n(\nu,P)-P+Y_n(\nu,P)
\end{equation*}
and $\varepsilon_n(a_n+ \log n) \to 0$ as $n \to 0$,
we have the result from Proposition \ref{Proposition2.6} and Proposition \ref{Proposition2.8}.
\end{proof}
\begin{proof}[Proof of Theorem \ref{Theorem2.11}]
Let $\varepsilon >0$ be taken arbitrarily. Then there exists a polynomial $P$ such that
\begin{equation}\label{eq3.16}
 \left\|\left(f(x)-P(x)\right)(1+|x|)^{\alpha}\left\{\Phi^{-\frac{3}{4}}(x)
  w(x)\left(|x|+\frac{a_n}{n}\right)^{\rho}\right\}^{\nu-\frac{\eta}{2}}\right\|_{L_{\infty}(\mathbb{R})}<\varepsilon.
\end{equation}
In fact, by (\ref{eq2.14}) we see
\begin{equation*}
\lim_{|x|\to \infty} \left|f(x)w^{\nu-\eta}(x)\right| =0,
\end{equation*}
hence as \cite[(4.55)]{[7]} we obtain (\ref{eq3.16}). Then we also have
\begin{equation}\label{eq3.17}
 \left\|\left(f(x)-P(x)\right)(1+|x|)^{\alpha}\left\{\Phi^{-\frac{3}{4}}(x)
  w_{\rho}(x)\right\}^{\nu}\right\|_{L_{\infty}(\mathbb{R})}<C\varepsilon,
\end{equation}
and from
\begin{equation*}
 \left\|(1+|x|)^{\Delta-\alpha}\left\{\Phi^{-\frac{3}{4}}(x)
  w(x)\left(|x|+\frac{a_n}{n}\right)^{\rho}\right\}^{\nu/2}\right\|_{L_{\infty}(\mathbb{R})}< \infty,
\end{equation*}
with (\ref{eq3.16}), we have
\begin{equation}\label{eq3.18}
 \left\|\left(f(x)-P(x)\right)(1+|x|)^{\Delta}\left\{\Phi^{-\frac{3}{4}}(x)
  w(x)\left(|x|+\frac{a_n}{n}\right)^{\rho}\right\}^{\nu}\right\|_{L_{\infty}(\mathbb{R})}<C\varepsilon.
\end{equation}
Now, we see
\begin{equation*}
\left|L_n(\nu,f)-f\right| \le \left|X_n(\nu,f-P)\right|+\left|X_n(\nu,P)-P\right|
+ \left|P-f\right|+ \left|Y_n(\nu,f)\right|.
\end{equation*}
From Proposition \ref{Proposition 2.5}  with $C_{f-P}=C\varepsilon$ and (\ref{eq3.17}), we know
\begin{equation} \label{eq3.19}
\left\|(1+|x|)^{-\Delta}\left\{\Phi^{\frac{3}{4}}(x)w(x)\left(|x|+\frac{a_n}{n}\right)^{\rho}\right\}^{\nu}
X_n(\nu,f-P)\right\|_{L_p(\mathbb{R})}
\le C\varepsilon.
\end{equation}
Then, using  (\ref{eq3.18}), (\ref{eq3.19}),
Proposition \ref{Proposition2.9} and Proposition \ref{Proposition2.6},
we conclude
\begin{eqnarray*}
\left\|(1+|x|)^{-\Delta}\left\{\Phi^{\frac{3}{4}}(x)w(x)\left(|x|+\frac{a_n}{n}\right)^{\rho}\right\}^{\nu}
\left(L_n(\nu,f)-f\right)\right\|_{L_p(\mathbb{R})}
\le C\varepsilon.
\end{eqnarray*}
Hence, we have the result.
\end{proof}
\begin{proof}[Proof of Theorem \ref{Theorem2.12}]
If we split $\left|L_n(l,\nu,f)-f\right|$ as follows:
\[
\left|L_n(l,\nu,f)-f\right| \le \left|L_n(\nu,f)-f\right| + \left|Z_n(l,\nu,f)\right|,
\]
then the result is proved from Theorem \ref{Theorem2.11} and Proposition \ref{Proposition2.7}.
\end{proof}



\end{document}